\documentclass[sn-mathphys-num]{sn-jnl}% Math and Physical Sciences Numbered Reference Style 

\usepackage[english]{babel}
\usepackage{amsmath,amssymb,amsthm}

\theoremstyle{definition}
\newtheorem{theorem}{Theorem}[section]
\newtheorem{example}[theorem]{Example}

\newtheorem{definition}[theorem]{Definition}
\newtheorem{remark}[theorem]{Remark}
\newtheorem{algorithm}[theorem]{Algorithm}

\newenvironment{ALGORITHM}  % fom MinPoly
  {\goodbreak\bigskip\hrule\begin{algorithm}}
  {\end{algorithm}\hrule\goodbreak}

\newcommand\ie{\textit{i.e.}}
\newcommand\point{P}

\usepackage[T1]{fontenc}
\usepackage[all]{xy}
\usepackage[utf8]{inputenc}
\usepackage{adjustbox}
\usepackage{amsfonts}
\usepackage{amsmath}
\usepackage{amssymb}
\usepackage{booktabs} 
\usepackage{booktabs}       % pacchetto per poter mettere le tabelle
\usepackage{calligra}
\usepackage{color}
\usepackage{enumitem} % for parsep, itemsep...
\usepackage{fancybox}
\usepackage{graphicx}
\usepackage{listings}
\usepackage{subfig}
\usepackage{tikz-cd}
\usepackage{tikz}
\usepackage{ulem}
\usepackage{xcolor}
\usepackage{xspace}
\usepackage{array} % for fonts in tabular
\usetikzlibrary{arrows}
\usetikzlibrary{calc}

\newcommand{\Groebner}{Gr\"obner\xspace}
\newcommand{\ZeroCasesToDo}{\text{VanishingToDo}\xspace}
\newcommand{\NewZeroCases}{\text{NewVanishing}\xspace}

\newcommand{\A}{\mathcal{A}}
\DeclareMathOperator{\LPP}{LPP}
\newcommand{\MBLPP}{{\mathrm{MB}}}
\newcommand{\minG}{G_{\MBLPP}}
\DeclareMathOperator{\LM}{LM}
\DeclareMathOperator{\LC}{LC}

\newcommand \ideal[1] {\langle #1 \rangle}
\newcommand{\Ia}{\mathfrak{a}}  % gotico a
\newcommand{\Ib}{\mathfrak{b}}  % gotico b
\newcommand{\Ic}{\mathfrak{c}}  % gotico c
\newcommand{\Iz}{\mathfrak{z}}  % gotico z
\newcommand{\zeroes}{\mathbb{V}}  % VV
\newcommand \ginKA {\mathfrak{g}}
\DeclareMathOperator \GG{{\mathrm CGS}}
\newcommand \FinalCases {\GG}
\newcommand \FinalCasesI {{\GG}_I}
\newcommand \pointsA {\mathcal{A}}
\newcommand \evalpoint {\mathrm{eval}_{\point}}

\usetikzlibrary{arrows}
\tikzset{
v/.style={
  circle, draw, inner sep=2pt, minimum size=3pt, fill=black}
}

\newcommand{\CC}{\mathbb{C}}

\newcommand{\QQ}{\mathbb{Q}}

\renewcommand{\to}{\longrightarrow}

\def \lex{{\rm lex}}

\begin{document}

\title{A new iterative algorithm for comprehensive Gr\"obner systems%\thanks{During the preparation of this article
% the third author was supported by JSPS Grant-in-Aid for Early-Career Scientists (19K14493)}%\thanks{Grants or other notes
}


%\author{Anna Maria Bigatti \and Elisa Palezzato \and Michele Torielli }

\author*[1]{\fnm{Anna Maria} \sur{Bigatti}}\email{bigatti@dima.unige.it}

\author*[2]{\fnm{Elisa} \sur{Palezzato}}\email{elisa.palezzato@gmail.com}
%%\equalcont{These authors contributed equally to this work.}

\author*[3]{\fnm{Michele} \sur{Torielli}}\email{michele.torielli@nau.edu}
%%\equalcont{These authors contributed equally to this work.}

\affil*[1]{\orgdiv{Dipartimento di Matematica},
  \orgname{Universit\`a degli Studi di Genova},
  \orgaddress{\street{Via Dodecaneso 35}, \city{Genova},
    \postcode{16146},
    \state{},
    \country{Italy}}}

%\affil*[1]{\orgdiv{Department}, \orgname{Organization}, \orgaddress{\street{Street}, \city{City}, \postcode{100190}, \state{State}, \country{Country}}}

\affil*[2]{\orgdiv{Department of Mathematics}, \orgname{Hokkaido University}, \orgaddress{\street{Kita 10, Nishi 8, Kita-Ku}, \city{Sapporo}, \postcode{060-0810}, \state{Hokkaido}, \country{Japan}}}

\affil[3]{\orgdiv{Department of Mathematics and Statistics}, \orgname{Northern Arizona University}, \orgaddress{\street{801 S. Osborne Dr.}, \city{Flagstaff}, \postcode{86001}, \state{Arizona}, \country{USA}}}


%% \institute{Anna Maria Bigatti \at
%%               Dipartimento di Matematica, Universit\`a degli Studi di Genoa, Via Dodecaneso 35, 16146 Genova , Italy. \\
%%             %  Tel.: +123-45-678910\\
%%           %    Fax: +123-45-678910\\
%%              \email{bigatti@dima.unige.it}        %  \\
%% %             \emph{Present address:} of F. Author  %  if needed
%%            \and
%%            Elisa Palezzato \at
%%            Department of Mathematics, Hokkaido University, Kita 10, Nishi 8, Kita-Ku, Sapporo 060-0810, Japan. \\
%%            \email{palezzato@math.sci.hokudai.ac.jp}
%%            \and
%%            Michele Torielli \at
%%            Department of Mathematics, Hokkaido University, Kita 10, Nishi 8, Kita-Ku, Sapporo 060-0810, Japan. \\
%%            \email{torielli@math.sci.hokudai.ac.jp}}


%%%%%%%%%%%%%%%%%%%%%%%%%%%%%%%%%%%%%%%%%%%%%%%%%
%%    ABSTRACT
%%%%%%%%%%%%%%%%%%%%%%%%%%%%%%%%%%%%%%%%%%%%%%%%%
\abstract{

    A    Comprehensive Gr\"obner system for a parametric ideal $I$ in $K(A)[X]$ represents the
collection of all Gr\"obner bases of the ideals $I'$ in $K[X]$ obtained as the values of the
parameters $A$ vary in $K$.
 
The recent algorithms for computing them %comprehensive Gr\"obner systems
consider the corresponding ideal $J$ in $K[A,X]$, and are
based on stability of Gr\"obner bases of ideals under specializations
of the parameters.  Starting from a Gr\"obner basis of $J$, the
computation splits recursively depending on the vanishing of the
evaluation of some ``coefficients'' in $K[A]$.

In this paper, taking inspiration from the algorithm described by
Nabeshima, we create a new iterative algorithm to compute
comprehensive Gr\"obner systems.
We show how we keep track of the sub-cases to be considered, and how
we avoid some redundant computation branches using
``comparatively-cheap'' ideal-membership tests, instead of
radical-membership tests.
}


 %% https://cran.r-project.org/web/classifications/MSC.html

\maketitle

%%%%%%%%%%%%%%%%%%%%%%%%%%%%%%%%%%%%%%%%%%%%%%%%%%%%%%%%%%%%%%%%%%%%%%
%%%% \tableofcontents  % for work in progress
%%%%%%%%%%%%%%%%%%%%%%%%%%%%%%%%%%%%%%%%%%%%%%%%%%%%%%%%%%%%%%%%%%%%%%

\section{Introduction}\label{intro}

%%%%%%%%%%%%%%%%%%%%%%%%%%%%%%%%%%%%%%%%%%%%%%%%%%%%%%%%%%%%%%%%%%%%%%
%\subsection{An example}

%\begin{example}
%\anna{ OK Elisa 2022-11-4
  Let us start with a simple example to introduce the central object of this paper.
%}
Consider 
${\cal C}_1$, a circle with center $(0,0)$ and radius $1$,
and ${\cal C}_2$, 
a circle with center $(c,0)$ and radius $\sqrt{r}$.
What is their intersection as $c$ and $r$ vary in $\CC$?
%\begin{figure}[t]
%\centering
%% \begin{tikzpicture}[scale=0.7]
%% \draw (0,0) node[v,label=above:{$ (  0,  0) $}](1){};
%% \draw (2,0) node[v,label=above:{$ (  c,  0) $}](2){};
%% \draw (-0.5,0.1) node[label=below:{$ 1 $}](3){};
%% \draw [densely dashed](-1,0) -- (0,0);
%% \draw (2.5,0.1) node[label=below:{$ \sqrt{r} $}](4){};
%% \draw [densely dashed](2,0) -- (3.5,0);
%% \draw (0,0) circle (1cm);
%% \draw (2,0) circle (1.5cm);
%% \end{tikzpicture}
%\end{figure}
We can see that there are three cases:
if $c=0, \, r=1$ the two circles coincide,
if $c=0, \, r\ne1$ the two circles do not intersect,
if $c\ne0$ the two circles intersect in two points in $\CC^2$
(may be one double point).
%\end{example}

Algebraically speaking, 
%2023-06-22
the ideal   $I = \ideal{x^2 {+}y^2 {-}1, \; (x{-}c)^2 {+}y^2{-} r}$,
defined by the equations of $\mathcal{C}_1$ and $\mathcal{C}_2$,
  has a collection of different \Groebner bases depending on the
  values of the parameters.
  This is described by a \textit{Comprehensive Gr\"obner system}, which is a set of pairs
$\{(\mathcal{A}_1, G_1),\dots, (\mathcal{A}_l, G_l)\}$, 
 the \Groebner bases $G_i$ under the conditions $\mathcal{A}_i$.
For example, one such pair is
$\mathcal{A}_1 = \{(r,c)\in\CC^2 \mid c\ne0\}$
and its corresponding 
%\elisa{OK Elisa 2022-11-4; non sono certa per come erano scritte le righe se volessi tutto scritto o solo la parentesi}
$\lex$-Gr\"obner basis 
($\lex$ with $x>y>c>r$)
\\
\phantom.\hfil$\{\;
x^2  +y^2  -1,
\;\;
c x -\frac{1}{2} c^2  +\frac{1}{2} r -\frac{1}{2},
\;\;
 {c^2} {y^2}  +\frac{1}{4} c^4  -\frac{1}{2} r c^2  +\frac{1}{4} r^2  -\frac{1}{2} c^2  -\frac{1}{2} r +\frac{1}{4}
 \;\}
$

% STORIA
% \michele{ Anna OK 2022-10-13
   Comprehensive Gr\"obner systems for parametric ideals were introduced, constructed, and studied in 1992 by Weispfenning, see \cite{wei92}. After Weispfenning's work, Kapur in \cite{Kap95} introduced an algorithm for parametric Gr\"obner bases. However, there was no big development on comprehensive Gr\"obner systems for around ten years. On the other hand, around 2003, big developments were made by Kapur-Sun-Wang \cite{KSW2010}, Montes \cite{montes02}, Nabeshima \cite{Nabe2007}, Sato and Suzuki \cite{SS2003, SS2006}, and Weispfenning \cite{wei2003}.
 
Some of algorithms for computing comprehensive Gr\"obner systems are
based on stability of Gr\"obner bases of ideals under specializations.
Each algorithm has a different ``stability condition'' for monomial
bases.

In this paper, taking inspiration from \cite{Nabe2012}, we create a new iterative algorithm to compute comprehensive Gr\"obner systems.

%\michele{ OK Anna 2024-03-19
  In Section~\ref{sec:Notation}, we recall the definition of a comprehensive Gr\"obner system and we fix the notation and basic definitions needed for the rest of the paper.
  In Section~\ref{sec:Stability}, we describe the
  stability conditions   used in previous papers to construct comprehensive Gr\"obner systems involving only computations in $K[A,X]$ instead of doing Gr\"obner basis computations in $K(A)[X ]$.
  In Section~\ref{sec:Iterative}, we describe our new iterative approach to compute comprehensive Gr\"obner system and we prove its correctness.
  In Section~\ref{sec:Timings}, we illustrate timings and statistics of our algorithm on known examples.
  In Section~\ref{sec:Future}, we provide some ideas for further
  investigations and possible improvements.
  %}

%% this is like saying that we want to compute ``all
%% Gr\"obner bases'' of the ideal
%% $I := \langle  x^2 +y^2 -1,  (x-c)^2 +y^2- r \rangle \subseteq \CC[x,y]$
%% for all given constants $c$ and $r$ in $\CC$.

%% Now consider  $I := \langle f,g \rangle \subseteq \CC[c,r,  x,y]$
%% with $\tau$, an \textit{elimination ordering for $x,y$} or
%% \textit{block ordering} on $ \CC[c,r, x,y]$.

%% The reduced lex-Gr\"obner basis of $I$ is
%% $G  := \{
%% x^2  +y^2  -1,
%% \,c x -\frac{1}{2} c^2  +\frac{1}{2} r -\frac{1}{2},
%% \,
%% rx {-1} {x} -2 c y^2  -\frac{1}{2} c^3  +\frac{1}{2} r c +\frac{3}{2} c,
%% \,
%% {c^2} {y^2}  +\frac{1}{4} c^4  -\frac{1}{2} r c^2  +\frac{1}{4} r^2  -\frac{1}{2} c^2  -\frac{1}{2} r +\frac{1}{4}
%% \;\}.$

%%   That means  that it is a lex  GB of the specialization of $I$ in $\CC[x,y]$
%%   for any pair of values  ${c}\ne0$ and ${r-1}\ne0$.

%%   But,  what  happens if ${c}$ or ${r-1}$ is/are 0?

  %% A questo punto darei la definizione di C G system con le definizioni
  %% (rapportate all'esempio perche' non rimangano astratte) 

%%%%%%%%%%%%%%%%%%%%%%%%%%%%%%%%%%%%%%%%%%%%%%%%%%%%%%%%%%%%%%%%%%%%%%
\section{Notation}\label{sec:Notation}

%% \anna{Cruciale: noi solitamente parliamo in termini di $I$, mentre SS
%%   parlano di $F$ (KSW?, N?).  E'
%%   delicato: cerchiamo di capire se ci sono trappole, per esempio LPP(F)
%%   diverso da LPP(I)}

%% \anna{2023-08-23 direi tutto OK: vedi promemoria $\ideal{\phi(F)} = \phi(\ideal F ) $}
%%%%%% promemoria Anna  
% phi : R1 -> R2 surjective homomorphism.   phi(<F>) = <phi(F)>
% (1) in (2):   phi(sum(hi*fi)) = sum( phi(hi)*phi(fi) )  in <phi(F)>
% (2) in (1):   sum(Hi*phi(fi)) = sum( phi(hi)*phi(fi) ) = phi(sum(hi*fi))
% In particular, phi(I) = <phi(I)>

  Let $K$ be a field, $\bar K$ its algebraic closure,  $X$ a set of $N_X$ indeterminates,
and $A$ a set of $N_A$ parameters.
We work in the polynomial ring
$K[A,X] = K[a_1, \dots, a_{N\!_A}, \; x_1,\dots,x_{N\!_X}]$.

  We want to study what happens when we specialize the
  parameters $A$.

      \begin{definition}
      Fixed a point $\point\in\bar K^{N\!_A}$,
      we indicate with     $\evalpoint(f)$
      the evaluation homomorphism  from $K[A]$ to~$\bar K$,\ $f\mapsto f(P)$.
Then we  also indicate with $\evalpoint$ its natural ``coefficientwise'' extension
$ K[A,X]\rightarrow \bar K[X]$,\ $f\mapsto f(P,X)$.
      \end{definition}

%%   so, for a given point $\point \in\bar K^{N\!_A}$   and
%% $f \in K[A,X]$, % \{g_{1},\ldots,g_{s}\}
%%   we denote by
%%   $\evalpoint(f)$  the evaluation  $f(\point ,X)\in \bar K[X]$.
     For example,  given the point $\point =(1,3)\in\QQ^{2}$ and
%%$f = 3axy^2 -bxy^2 +ax^2 +bxy -xy +1 \in \QQ[a,b,\;x,y]$,
$f = 3axy^2 -bxy^2 +ax^2 +b $ in $ \QQ[a,b,\;x,y]$,
we have
$\evalpoint(f) % = 3xy^2 -3xy^2  +x^2 +3xy -xy +1
%%= x^2 +2xy +1 \in \QQ[X]$.
= x^2 +3 \in \QQ[X]$.

%% \cancella{
%% Then we extend this notation to a set $F\subseteq K[A,X]$
%%   as   $\evalpoint(F) := \{\evalpoint(f) \mid f\in F\}
%%    \subseteq \bar K[X]$.
%% }

   Fixed a term-ordering $\tau$
   %%for $X$ \michele{intendi elimination term -ordering for $X$?}
   on the power-products of
  $K[A,X]$, for a polynomial $f\in K[A,X]$
  we denote by $\LPP(f)\in K[A,X]$ the $\tau$-greatest power-product
  in the support of~$f$, by $\LC(f)\in K$ its coefficient, and
  by $\LM(f) = \LC(f)\cdot \LPP(f)$.
%  Usually there will be no ambiguity on the ordering $\tau$, so we'll
%  just write $LPP(f)$, $LC(f)$, $LM(f)$.

  %\anna{ OK Elisa 2022-11-4
  Throughout this paper we will only consider 
%$K[A, X]\cong(K[A])[X]$
\textit{elimination orderings}
 for~$X$, \ie~such that any power-product only in $A$
is smaller than all power-products containing some
%\anna{ OK Elisa 2022-11-4
$x_i\in X$. %}
  Such orderings reflect the interpretation of $K[A,X]$ as
  $K[A][X]$, with the polynomials in $A$
  seen as ``parametric coefficients'':

    \begin{definition}
Given on $K[A,X]$ an elimination ordering $\tau$ for $X$, we indicate
with~$\tau_X$ its restriction to the power-products only in~$X$.
Then we indicate  with $\LPP_X(f)\in K[X]$
the $\tau_X$-greatest power-product of~$f$ as a polynomial in $K[A][X]$,
and with
  $\LC_X(f)$ its coefficient in $K[A]$.
%, and $LM_X(f) = LC_X(f)\cdot LPP_X(f)$.
%\michele{Lo vogliamo in una definizione numerata?}    OK
    \end{definition}

   \begin{example}
     For $f = (a^2{-}b)x^3y + 3abxy^2\in \QQ[a,b,x,y]$ equipped
     with any elimination ordering for $X=\{x,y\}$
     such that $x^3y >xy^2$, e.g. $\tau=\lex$,
    % \anna{
       then we have $\LPP(f)=a^2x^3y$ with $\LC(f)=1$, and
      $\LPP_X(f)=x^3y$ with $\LC_X(f)=a^2{-}b$.
   %  }
   \end{example}

  Throughout this paper, we indicate with gothic letters $\Ia$, $\Ib$,
  $\ginKA$ ideals in $K[A]$ (and their embedding in $K[A,X]$), and with
  $I$, $J$ ideals in $K[A,X]$.

   Now we can formally define a \textit{Comprehensive Gr\"obner system}.

%\cancella{
%\begin{definition}
%  Given a set of polynomials $F\subseteq K[A, X]$,
%  and a finite set of pairs, called \textit{segments}, $\GG$ =
%$\{(\mathcal{A}_1, G_1),\dots, (\mathcal{A}_\ell, G_\ell)\}$, where
% \begin{itemize}
% \item $\mathcal{A}_1,\dots, \mathcal{A}_\ell$ are
%   \textit{algebraically constructible} subsets of $\bar K^{N\!_A}$,
%   \ie~$\zeroes(\Ia_j)\setminus\zeroes(\Ib_j)$, with $\Ia_j,\Ib_j $
%   ideals in $K[A]$
%  
%\item $G_1,\dots, G_\ell$ finite subsets of $K[A, X]$
% \end{itemize}
%then $\GG$ is a
%  \textit{comprehensive Gr\"obner system} on $\cup\mathcal{A}_i$ for
%  $F$
%  if
%  for each \textit{segment} $(\A_i,G_i)\in\GG$
%    %$i=1,\dots, \ell$
%    we have that
%for all  $\bar a\in\mathcal{A}_i$,
%  $\evalpoint(G_i)$ is a $\tau_X$-Gr\"obner basis of the
%  ideal $\langle  \evalpoint(F)\rangle$ in~$\bar K[X]$.
% %Each $(\mathcal{A}_i, G_i)$ is called a \OR{segment} of $\mathcal{G}$. 
%\end{definition}
%}

\begin{definition}
 % \cancella{  Given a set of polynomials $F\subseteq K[A, X]$,}
  Given an ideal  $I$ in $K[A, X]$,
    a finite set of pairs, called \textbf{segments}, 
$\GG=\{(\mathcal{A}_1, G_1),\dots, (\mathcal{A}_\ell, G_\ell)\}$
 is a   \textbf{comprehensive Gr\"obner system} for   $I$ if
  
 \begin{itemize}
\item $\mathcal{A}_1,\dots, \mathcal{A}_\ell$ are
   \textit{algebraically constructible} subsets of $\bar K^{N\!_A}$
\\   (\ie~$\mathcal{A}_j=\zeroes(\Ia_j)\setminus\zeroes(\Ib_j)$  with $\Ia_j,\Ib_j $
   ideals in $K[A]$)
   and $\cup\mathcal{A}_j = \bar K^{N\!_A}$
  
\item $G_1,\dots, G_\ell$ are finite subsets of $K[A, X]$
 
\item
  for each \textit{segment} $(\A_j,G_j)$ in $\GG$    %$i=1,\dots, \ell$
    and  for any point   $\point\in\mathcal{A}_j$ we have that
  $\evalpoint(G_j)$ is a $\tau_X$-Gr\"obner basis of the
  ideal $\evalpoint(I)$ in~$\bar K[X]$.
 %Each $(\mathcal{A}_j, G_j)$ is called a \OR{segment} of
  %$\mathcal{G}$.
  
 \end{itemize}
% }

%%  \cancella{
%%    $\GG$ is a   \textit{comprehensive Gr\"obner system} for   $F$ if
  
%%  \begin{itemize}
%% \item
%%    $\GG$ is a finite set of pairs, called \textit{segments}, 
%% $\{(\mathcal{A}_1, G_1),\dots, (\mathcal{A}_\ell, G_\ell)\}$

%%  \begin{itemize}
%% \item $\mathcal{A}_1,\dots, \mathcal{A}_\ell$ are
%%    \textit{algebraically constructible} subsets of $\bar K^{N\!_A}$
%%    (\ie~$\mathcal{A}_j=\zeroes(\Ia_j)\setminus\zeroes(\Ib_j)$  with $\Ia_j,\Ib_j $
%%    ideals in $K[A]$)
%%    and $\cup\mathcal{A}_j = \bar K^{N\!_A}$
  
%% \item $G_1,\dots, G_\ell$ are finite subsets of $K[A, X]$
%%  \end{itemize}
 
%% \item
%%   for each \textit{segment} $(\A_j,G_j)$ in $\GG$    %$i=1,\dots, \ell$
%%     and  for any point   $\point\in\mathcal{A}_j$ we have that
%%   $\evalpoint(G_j)$ is a $\tau_X$-Gr\"obner basis of the
%%   ideal $\langle  \evalpoint(F)\rangle$ in~$\bar K[X]$.
%%  %Each $(\mathcal{A}_j, G_j)$ is called a \OR{segment} of
%%   %$\mathcal{G}$.
  
%% \end{itemize}
%% }

\end{definition}

%%%%%%%%%%%%%%%%%%%%%%%%%%%%%%%%%%%%%%%%%%%%%%%%%%%%%%%%%%%%%%%%%%%%%%
\section{Stability and algorithms}\label{sec:Stability}

  In this section we recall how recent algorithms use stability
  conditions to   construct comprehensive Gr\"obner systems
involving only computations in $K[A,X]$ instead of having to
track the arithmetic inside a Gr\"obner basis computation in
$K(A)[X]$.

%% \begin{definition}
%%   Let $\tau$ be an~$X$-elimination term-ordering on the power-products
%%   in $K[A,X]$.
%%     An ideal $I$  in $K[A,X]$ 
%%    is \textbf{stable} under $\evalpoint$
%%   and the term-ordering~$\tau_X$ if
%%   $\ideal{\evalpoint(\LM_X(I))}=\LM_X(\evalpoint(I))$.  (?? LM X  per
%%   entrambi ??)
%%   %% In other words, for any  $G$ $\tau$-Gr\"obner basis of $I$,
%%   %% $\evalpoint(G)$ is a $\tau_X$-Gr\"obner basis of $\evalpoint(I)$.
%% \end{definition}

%% \cancella{
%%   Considering
%%   as $\phi(f)$ the evaluation in a point of $\bar K^{N\!_A}$, $\evalpoint(f)$,
%%   Sato and Suzuki \cite{SS2003, SS2006} 
%% designed the first (recursive) algorithm for
%% finding the segments in the comprehensive Gr\"obner system of $I$
%% involving only computations in $K[A,X]$, instead of having to
%% track the arithmetic inside a Gr\"obner basis computation in~$K(A)[X]$.
%% }

Suzuki-Sato in \cite{SS2006}
observed that the  stability theorem
proved by  Kalkbrener (for any
ring homomorphism $\phi:K[A]\to\bar K$),
re-read  in terms of $\evalpoint$,
provided this new approach for computing comprehensive Gr\"obner systems.

\begin{theorem}[Kalkbrener \cite{Kal1997}, Suzuki-Sato \cite{SS2006}]\label{thm:Kal-SS}
  Let $I$ be an ideal in $K[A,X]$.
  Let $\tau$ be an~$X$-elimination term-ordering on the power-products
  in $K[A,X]$
  and $G$  a $\tau$-Gr\"obner basis of  $I$.
%  and $\ginKA = G\cap K[A]$.

Fix a point $\point\in\bar K^{N\!_A}$ such that
$\evalpoint(\LC_X(g)){\ne}0$ for each $g \in G\setminus(G{\cap} K[A])$.
Then, % $I$ is stable under $\evalpoint$ and the term-ordering $\tau_X$,
$\evalpoint\{G\}$ is a $\tau_X$-Gr\"obner basis of  $\evalpoint(I)$.
\end{theorem}
%}
  
%%   \cancella{%%%%%% cancella %%%%%%%%%%%%%%%%%%%%%%%%%%%%%%%
%%     \begin{theorem}\label{thm:Kal97}
%%   (Kalkbrener \cite{Kal1997}).
%%   Let $\tau$ be an~$X$-elimination term-ordering on the power-products
%%   in $K[A,X]$.
%%   Let $\phi$ be a ring homomorphism  from $K[A]$ to~$\bar K$,
%%   and also indicate with $\phi$ its natural ``coefficientwise'' extension
%%   $\phi\colon K[A,X]\rightarrow \bar K[X]$.

%%     Let $I$ be an ideal in $K[A,X]$ and
%% let $G =\{g_1,\dots,g_s\}$ be  a $\tau$-Gr\"obner basis of  $I$
%% indexed so that,  for some $r$, we have
%% $\phi(\LC_X(g_i))  \begin{cases}
%%      \ne 0& \text{if }i= 1,...,r\\
%%      = 0& \text{if }i=r{+}1,...,s
%%   \end{cases}$
%% \\Then, the following three conditions are equivalent.

%% \begin{enumerate}
%% \item $I$ is stable under $\phi$
%%   %\michele{mi sembrava di aver gia' scritto la definizione di
%%   %stabile}
%%   and the term-ordering $\tau_X$, i.e.
%%   $\ideal{\phi(\LM_X(I))}=\LM_X(\phi(I))$.

%% \item $\{\phi(g_1),...,\phi(g_r)\}$ is a $\tau_X$-Gr\"obner basis of  $\phi(I)$.

%% \item For every $i \in \{r{+}1,...,s\}$, $\phi(g_i)$ is reducible to
%%   0 modulo $\{\phi(g_1),\dots,\phi(g_r)\}$   in $\bar K[X]$.
%% \end{enumerate}
%% \end{theorem}
%%   }

  We briefly illustrate with a small example how they applied this
theorem for constructing a comprehensive Gr\"obner system.

\begin{example}\label{ex:two-circles}
  Consider again, from the Introduction,
  $I := \langle f,g \rangle \in \CC[c,r, x,y]$,
  where $f := x^2 +y^2 -1$ and  $g := (x{-}c)^2 +y^2- r$,
  and we fix the term-ordering $\lex$ with $x>y>c>r$, which is an elimination ordering for $\{x,y\}$.
The first step of the algorithm consists in computing
%with $\tau$ an \textit{elimination ordering for $x,y$} \;\; or
%\textit{block ordering}
%\\
%(\textit{lex} on $x,y$)
a $\lex$-reduced Gr\"obner basis of $I$:
\\
$G  = \{\underline{\;x^2}  +y^2  -1, \quad
\underline{\;{c} {x}\;} -\frac{1}{2} c^2  +\frac{1}{2} r -\frac{1}{2},$
\\\phantom.\hfill
$\underline{\;{r} {x} {-} {x}} -2 c y^2  -\frac{1}{2} c^3  +\frac{1}{2} r c
+\frac{3}{2} c, \quad
    \underline{\;{c^2} {y^2}}  +\frac{1}{4} c^4  -\frac{1}{2} r c^2  +\frac{1}{4} r^2  -\frac{1}{2} c^2  -\frac{1}{2} r +\frac{1}{4}
\;\}.$
%%  so   $h_1 = 1$, \; $h_2=c$, \; $h_3=r-1$, \; $h_4=c^2$

From Theorem~\ref{thm:Kal-SS} we know that $G$
is a $\lex_{x,y}$-Gr\"obner basis of
 the specialization of $I$ 
 for any $(c,r)$ in
 $\A_1 = \bar\QQ^2\setminus(\zeroes(c)\cup\zeroes(r{-}1)\cup\zeroes(c^2))$. 
%
%   \bigskip
   Then we recursively deal with the remaining zero cases
   $\zeroes(c)$, \;  $\zeroes(r-1)$, \; $\zeroes(c^2)$
by  recursively running the algorithm on the ideals 
      $I+\langle {c}\rangle$, \;  $I+\langle {r{-}1}\rangle$,
   \; {$I+\langle {c^2}\rangle$}.
   
\end{example}

\begin{remark}\label{rem:radical}
Notice that  we can reduce the number of recursion steps by considering
  the radical of the ideals, in this example $\sqrt{\ideal{c^2}} =
  \ideal{c}$ (see also Section~\ref{sec:Avoiding}).
Unfortunately, computing the radical of an ideal is difficult
to implement and possibly quite slow.
\end{remark}

%\anna{
There is a delicate balance to keep in mind:
reducing the number of cases to be considered, while controlling the
cost for detecting them.

In this direction,  further advance was  achieved by proving
that just a subset of the Gr\"obner basis %$G$
in $K[A,X]$
  needs to be considered.
  In \cite{KSW2010} Kapur-Sun-Wang discovered a new stability
  condition, and
  in \cite{Nabe2012}, Nabeshima improved it, further reducing
  redundant computation branches.

They are both centered on the following definition. 
%}

%  \anna{...new stability conditions...  allineare gli enunciati?
%    Kalkbrener enunciato con eval?}

%%%%%%%%%  Kapur-Sun-Wang

%% %  \anna{spostare alcuni commenti presi dall'introduzione:}
%% \cancella{In 2010, in \cite{KSW2010} Kapur-Sun-Wang discovered a  new stability condition and constructed a new algorithm for computing comprehensive Gr\"obner systems. Kapur-Sun-Wang's stability condition was stronger and more efficient than Suzuki-Sato's condition described in \cite{SS2006} and Nabeshima's one from \cite{Nabe2007}.
%% In \cite{Nabe2012}, Nabeshima improved Kapur-Sun-Wang's algorithm by introducing a new strong stability condition. The main advantage of Nabeshima's approach is that it generates fewer segments compared to Kapur-Sun-Wang's algorithm. Notice that all the previous algorithms are all recursive ones. 
%% }

\begin{definition}\label{def:MB}
For a finite set $F\subseteq K[A,X]$
we define {\boldmath$\MBLPP(F)$} to be the
\textbf{Minimal Basis} of the monomial ideal $\ideal{\LPP_X(f) \mid f\in F} \;
\subset K[X]$
(\textit{e.g.}~in  Example~\ref{ex:two-circles} we have $\MBLPP(G)=\{x,y^2\}$).
\end{definition}

%\anna{usare t invece che p}

\begin{theorem}\label{thm:KSW-N}
  Let $I$ be an ideal in $K[A,X]$.
  Let $\tau$ be an~$X$-elimination term-ordering on the power-products
  in $K[A,X]$
and $G$  a $\tau$-Gr\"obner basis of  $I$.

Let $\ginKA = G\cap K[A]$ and
name $\{t_1, \dots, t_s\}$ the power-products in $\MBLPP(G\setminus \ginKA)$.

\begin{description}
  \item[(Kapur-Sun-Wang \cite{KSW2010})]\label{thm:KSW10}
    For each~$t_i$, select one polynomial $g_{t_i}{\in} G$
    (not uniquely determined) such that
    $\LPP_X(g_{t_i})=t_i$ 

  Then, for all $\point \in
\zeroes(\ginKA)\setminus(\zeroes(\LC_X(g_1))\cup \zeroes(\LC_X(g_2)) \cup
\cdots\cup \zeroes(\LC_X(g_s)))$,\\
$\evalpoint(\{g_{t_1},\ldots,g_{t_s}\})$
is a $\tau_X$-Gr\"obner basis of $\evalpoint(I)$ in $\bar{K}[X]$.

%% By Theorem~\ref{theo:Kal97} and Theorem~\ref{theo:KSW10}, we can
%% easily obtain the following corollary.
%% \begin{Corollary}
%% With the same notations in Theorem~\ref{theo:KSW10}, then, a stability condition of 
%% $\LPP(\Noncomparable(G\setminus \ginKA))$ is
%% $\zeroes(\ginKA)\setminus(\zeroes(\LC_X(g_1))\cup \zeroes(\LC_X(g_2))
%% \cup \cdots\cup \zeroes(\LC_X(g_s)))$.
%% (Clearly, in $K[A][X]$, $\LPP(\Noncomparable(G\setminus \ginKA))$ is a specialized minimal leading monomial basis of $I$.)
%% \end{Corollary}

%% \begin{theorem}  (Nabeshima \cite{Nabe2012}).
%%   Let $I$ be an ideal in $K[A,X]$.
%%   Let $\tau$ be an~$X$-elimination term-ordering on $K[A,X]$
%% and $G$  a $\tau$-Gr\"obner basis of  $I$.
 
%% Let $\ginKA = G\cap K[A]$ and
%% name $\{t_1, ..., t_s\} = \MBLPP(G\setminus \ginKA)$.
%% Then, 

\item[(Nabeshima \cite{Nabe2012})]\label{thm:N2012}
  For each~$t_i$, let $G_{t_i} =\{g\in G \mid \LPP_X(g) = t_i\}$.

  Then, for all
$\point \in\zeroes(\ginKA)\setminus
(   \zeroes(\LC_X(G_{t_1})
\cup\zeroes(\LC_X(G_{t_2})
\cup \cdots \cup \zeroes(\LC_X(G_{t_s}))$,\\
$\evalpoint(G_{t_1} \cup G_{t_2} \cup \cdots\cup G_{t_s})$
is a $\tau_X$-Gr\"obner basis  of $\evalpoint(I)$
 in $\bar{K}[X]$.
\end{description}

\end{theorem}
%%%%%%%%%%%%%%%%%%%%%%

  \begin{example}
    Consider again the two circles from Example~\ref{ex:two-circles}:
    we have $\MBLPP(G) = \{x, y^2\}$.
    So by Theorem~\ref{thm:N2012}(Nabeshima) we
    have $\ginKA = \ideal{0}$,
     \\ $G_{x} =\{\underline{{c} {x}} -\frac{1}{2} c^2  +\frac{1}{2} r -\frac{1}{2}, \quad
\underline{{r} {x} {-1} {x}} -2 c y^2  -\frac{1}{2} c^3  +\frac{1}{2} r c +\frac{3}{2} c\}$, 
     \\ $G_{{y^2}} =\{ \underline{{c^2} {y^2}}  +\frac{1}{4} c^4  -\frac{1}{2} r
c^2  +\frac{1}{4} r^2  -\frac{1}{2} c^2  -\frac{1}{2} r
+\frac{1}{4}\}$.

Then, for all
$\point \in\bar\QQ\setminus(\zeroes(\ideal{c,r-1}) \cup\zeroes(\ideal{c^2}))$,
$\evalpoint(G_{{x}} \cup G_{{y^2}})$
is a $\lex_{\{x,y\}}$-Gr\"obner basis  of $\evalpoint(I)$
in $\bar{\QQ}[x,y]$.

  \end{example}

%% (Remark that, for $V := \zeroes(f_1,\dots,f_s)$ and $W :=
%% \zeroes(g_1,\dots,g_t)$, $V \cup  W = \zeroes(f_ig_j : i=1,\dots,
%% s, j=1,\dots,t)$.)

%% %%%%%%%%%%%%%%%%%%%%%%%%%%%%%%%%%%%%%%%%%%%%%%%%%%%%%%%%%%%%%%%%%%%%%%
%% \subsection{The recursive algorithm in $K[A,X]$}

%% In this section we recall the papers and the algorithms inspiring us:
%% Suzuki-Sato, Kapur-Sun-Wang, Nabeshima

  %(the sets A_i are given by the conditions 1st <> 0 and g_inKA = 0)

%\begin{center}
%\run{demo.cocoa5}{\FG{Demo 1}}
%\end{center}
% illustrate the algorithm using SetVerbosityLevel  

%%%%%%%%%%%%%%%%%%%%%%%%%%%%%%%%%%%%%%%%%%%%%%%%%%%%%%%%%%%%%%%%%%%%%%
\subsection{Avoiding redundancy}\label{sec:Avoiding}

%\anna{
In Remark~\ref{rem:radical}
we  pointed out that we may compute \textit{radicals} of the ideals $\Ia$ in $K[A]$
because they describe the same set of points in $\bar K^{N\!_A}$, and,
by doing so, we likely avoid some inconclusive intermediate step in
the algorithm.
The Gr\"obner basis computation of $I+\sqrt{\Ia}$
should generally be easier (but not always so).
The problem here is that computing $\sqrt{\Ia}$ is hard to
implement and may be quite slow to run, so the cost of such computation could be
higher than the saving it produces.

In \cite{Nabe2012} redundancy is avoided by checking, before computing
the Gr\"obner basis of $I+\Ia \in K[A,X]$, whether or not $\zeroes(\Ia)$ 
actually contains some new point % in $\bar K ^{n_A}$
not yet been considered in the previous recursive calls.
The information of the processed sets of points
%\cancella{ points already   considered , }
is passed through the recursive calls in the form of one ideal: in fact,
note that
$\zeroes(\Ib_1)\cup \zeroes(\Ib_2) = \zeroes(\Ib_1 \cdot  \Ib_2)$,
so the union of all zero sets  may be represented by just one
ideal $\Ib$.
Then, by checking that  $\zeroes(\Ia)\setminus\zeroes(\Ib)\ne\emptyset$,
we guarantee that no point in $\bar K^{N_A}$ is considered twice.
This condition, equivalent to $\Ia\not\subseteq \sqrt{\Ib}$,
can be computed by \textit{radical membership} without
actually computing
$\sqrt \Ib$:
e.g. $f \not\in\sqrt{\Ib}$
is equivalent to  $\Ib+\ideal{f{-}1}\ne\ideal1$.
Therefore, radical membership may be computed in any CAS and is
generally cheaper than computing radicals.  However,  it is 
still quite expensive.

\smallskip

  Here we propose a transversal approach:
  we convert the  recursive structure of the algorithm into an iteration,
  collecting all the ideals yet to be considered into one set,
  \ZeroCasesToDo.
The fact of  having all ideals in one
  set, allows to apply radicals and/or radical membership if desired,
  but in addition provides further criteria based on simple ideal membership,
  which is faster to compute (see Remark~\ref{rem:iter-minimalization}).

%%%%%%%%%%%%%%%%%%%%%%%%%%%%%%%%%%%%%%%%%%%%%%%%%%%%%%%%%%%%%%%%%%%%%%
\section{The iterative algorithm}\label{sec:Iterative}

In general, it is always worth investigating % considering
whether a recursion could be
translated into an iteration. Thinking this way may produce new insights
and more efficient code.
Instead of recursively passing the objects to be dealt with,
we can collect them into one set (we call it \ZeroCasesToDo), and one
by one we pick and process them.

This approach offers two general optimizations.
One is that we can choose how to keep this set ``minimal'' (where the
meaning of ``minimal'' may vary), and the other is that we can choose
which ``object'' to pick first.

Inspired by
%Starting from the recursive algorithm  designed by
Nabeshima's \cite{Nabe2012} recursive algorithm, we %wrote
designed this iterative algorithm.
We keep \ZeroCasesToDo ``minimal'' in the sense of \textit{ideal
  inclusion} (see Theorem~\ref{thm:minimalized}).

\begin{ALGORITHM}
  \textsc{CGS-Iter}
  \label{GCS:alg:CGS-Iter}
  \begin{description}[topsep=0pt,parsep=1pt]
  %% \item[\textit{Notation:}]
  %%   Let $R =K[A,X] = K[a_1, \dots, a_{N\!_A}, \; x_1,\dots,x_{N\!_X}]$
  %%   and $\tau$ an $X$-elimination ordering on $R$.

  \item[Input] $I$ ideal $\subset R=K[A,X]$
     with an $X$-elimination ordering.

  \item[1] \textit{Initialization:}

\ZeroCasesToDo := $\{\ideal{0}\}$ \hfill(set of ideals in $K[A]$)\\
\; $\FinalCases := \emptyset$ \hfill(set of pairs $(\pointsA, G)$, with
$\pointsA{\subset} \bar K^{N\!_A}$ and $G \subset R$)

 \item[2] \textit{Main Loop:} \; 
while \; \ZeroCasesToDo $\ne\emptyset$ \; do

\begin{description}[topsep=0pt,parsep=1pt]
\item[2.1] $\Ia$ := choose and remove one ideal from \ZeroCasesToDo

\item[2.2] $J := I + \Ia$

\item[2.3] Compute $G$, a $\tau$-Gr\"obner basis of $J$

\item[2.4] $\ginKA := J \cap K[A]$
\hfill  \textit{($\longrightarrow$ Note: $\ginKA$ is generated by $G \cap K[A]$,
    \,and\, $\Ia \subseteq \ginKA$)}

\item[2.5] if \; $\zeroes(\Ia) {\setminus} \zeroes(\ginKA)\ne\emptyset$
  \; then
  \hfill  \textit{($\longrightarrow$ Note: $\Ia \subsetneq \ginKA$)}

\begin{description}[topsep=0pt,parsep=1pt]
\item[2.5.1] append  the pair
  $(\zeroes(\Ia){\setminus} \zeroes(\ginKA), \; \{1\})$
  \;  to $\FinalCases$ 
\item[2.5.2] if \;$\ginKA$ contains no ideal in \ZeroCasesToDo \; then
  %  \\\phantom.
  \hfill  append $\ginKA$  to \ZeroCasesToDo 
\end{description}
% \spz\spz $\ginKA  := \{ g \in \GB(J) \mid g\in K[A] \};$

%\spz\spz $rad := \sqrt{\ideal{g_{\in K[A]}}};$

%% if radg_inKA, J)) then
%%       append(ref \ZeroCasesToDo, radg_inKA);
%%       continue;  -- "go to beginning of next loop" (5Jul)
%%     endif;

\item[2.6] else
  \; \hfill\textit{($\longrightarrow$ Note:  $\zeroes(\Ia) = \zeroes(\ginKA)$)}

  \begin{description}[topsep=0pt,parsep=1pt]

    \item[2.6.1]
      $\MBLPP$ := $\MBLPP(G\setminus\ginKA)$
      %, the minimal basis of  $\ideal{\LPP_X(g) \mid g\in      G\setminus\ginKA}      \;
        $\subseteq K[X]$
    \item[2.6.2]
      $\minG := \{ g \in G \,\mid\, \LPP_X(g) \in \MBLPP \} \subseteq G$
    \item[2.6.3]
%%      $minLC := \{ \ideal{LC_X(g) \mid \LPP_X(g) = t} \mid t\in \MBLPP\}$
      foreach $t\in \MBLPP$ \; let \;
      $\Ic_{t} := \ideal{\LC_X(g) \mid \LPP_X(g) = t} \;\subseteq K[A]$
    \item[2.6.4]
      append         the pair
%      $(\;\zeroes(\Ia) \setminus
      $(\;\zeroes(\ginKA) \setminus
      (\bigcup_{t\in \MBLPP} \zeroes(\Ic_t) ), \;\; \minG)$
      \;to       $\FinalCases$
      
    \item[2.6.5]
      \NewZeroCases := minimalized
      $\{ \Ic_t + \ginKA \mid
      t \in \MBLPP  \text{ and } \Ic_t\not\subseteq\ginKA\}$

    \item[2.6.6]
      for each $\Ib$ in \NewZeroCases do
      %   \begin{description}[topsep=0pt,parsep=1pt]   \item
      \\
      if \;$\Ib$ contains no ideal in \ZeroCasesToDo \; then
      %\\\phantom.
      \hfill	append $\Ib$  to \ZeroCasesToDo 
%  \end{description}
  \end{description}
\end{description}
\item[Output] {return} $\FinalCases$;
\end{description}
\end{ALGORITHM}

\begin{theorem}\label{thm:algorithm}
    Let $R =K[A,X] = K[a_1, \dots, a_{N\!_A}, \; x_1,\dots,x_{N\!_X}]$
    and $\tau$ and $X$-elimination ordering for the power-products in $R$.
    Let  $I$ be any ideal in $R$.  Then
    \begin{enumerate}
    \item\label{alg:terminates}
      The algorithm  \textsc{CGS-Iter} with input $I$ terminates.\quad
      Call     $\FinalCasesI$ \ its output.
    \item\label{alg:correctness}
      For each pair $(\pointsA, G)$ in $\FinalCasesI$,
      $\evalpoint(G)$ is a $\tau_X$-Gr\"obner basis for
      $\evalpoint(I)$   for any      point $\point\in\pointsA$. 

    \item  \label{alg:ConsideriamoTuttiIPunti}
      Each point $\point\in K^{N\!_A}$  is in  (at least)  a set
      $\pointsA$ in $\FinalCasesI$.
  \end{enumerate}
    In summary, the output
    $\FinalCasesI$ \ is a $\tau$-comprehensive Gr\"obner system for $I$.
\end{theorem}

\begin{proof}
  (\ref{alg:terminates})
  Each ideal $\Ib$ added to  \ZeroCasesToDo
%  (either in Step~[2.5.2] or in Step~[2.6.6])
   contains $\ginKA = (I+\Ia)\cap K[A]$,
  where $\Ia$ is its ``parent''
  %from \ZeroCasesToDo
  (Step~[2.1]).  Let's see that $\Ib\supsetneq \Ia$.
  If added in Step~[2.5.2] we have $\Ib=\ginKA\supsetneq\Ia$
  because $\zeroes(\Ia) {\setminus} \zeroes(\ginKA)\ne\emptyset$.
  If added in Step~[2.6.6] we have $\Ib\supsetneq\ginKA\supseteq\Ia$,
  because $\Ib = \Ic_t + \ginKA$ where $\Ic_t \not\subseteq \ginKA$.
  In conclusion, we build strictly increasing chains of ideals which
  are finite by the noetherianity of $K[A]$.
  Moreover, there are finitely
  many such chains because \NewZeroCases in Step~[2.6.6] is finite,
  as $\MBLPP$ is finite for the noetherianity of $R$.
  %% \anna{basta davvero per dire tutti i casi da trattare sono finiti?
  %% non mi sembra perche' (anche se non lo facciamo) potremmo aggiungere
  %% $10^i$ copie di $\Ib$ che sembra far esplodere il numero.  Forse
  %% dobbiamo considerare $\sum\ZeroCasesToDo$}
  
 % \anna{
  (\ref{alg:correctness})
If  $(\pointsA, G)$ was added 
    in Step 2.5.1, then  $\mathcal{A}= \zeroes(\Ia){\setminus} \zeroes(\ginKA)$,
    thus $(\mathcal{A}, \; \{1\})$
    is a correct segment because $\ginKA\subseteq I+\Ia$
    therefore for any point
      $\point\in\mathcal{A}$
    the ideal  $\evalpoint(I+\Ia)$ contains a non-zero    constant.

If  $(\pointsA, G)$ was added 
    in Step 2.6.4 then  $\mathcal{A} =
    \zeroes(\Ia) \setminus (\bigcup_{t\in \MBLPP} \zeroes(\Ic_t) )$,
     thus $(\mathcal{A}, \;\; \minG)$
     is a correct segment because  each point in $\mathcal{A}$
     satisfies the hypotheses of Theorem~\ref{thm:KSW-N}.Nabeshima
     for the ideal $I+\Ia$.
%  }

  (\ref{alg:ConsideriamoTuttiIPunti})
 Let's consider first what happens in a single iteration of the main loop.
Fix a point $\point\in \zeroes(\Ia)$ at Step~2.1.
  We claim that, before the next iteration, $\point$ is added %as a point in
to $\FinalCases$ or,   to be treated in a later iteration,
as a vanishing point in  \ZeroCasesToDo.

  If $\zeroes(\Ia) {\setminus} \zeroes(\ginKA) \ne \emptyset$    then
    \begin{itemize}
  \item
if $\point\in \zeroes(\Ia) {\setminus} \zeroes(\ginKA)  $
    then it is contemplated in $\FinalCases$ (added in Step~2.5.1).
  \item
    else $\point\in\zeroes(\ginKA)$.
    Then, $\ginKA$ is added to \ZeroCasesToDo,
    unless it contains some $\Ib$ in \ZeroCasesToDo.
    In either case, $\point$ will be listed for consideration as point of
    $\zeroes(\ginKA)$ or of $\zeroes(\Ib)\subset\zeroes(\ginKA)$.
  \end{itemize}
    Otherwise, $\zeroes(\Ia) \subset \zeroes(\ginKA)$.
\begin{itemize}
  \item
    If $\point\in \zeroes(\Ia) {\setminus} (\bigcup_{t\in \MBLPP} \zeroes(\Ic_t))  $
    then $\point$ is contemplated in $\FinalCases$ (added in Step~2.6.4)
  \item
    else $\point\in\zeroes(\Ic_t)$ for some $t\in \MBLPP$.
    Then, $\Ic_t$ is added to \ZeroCasesToDo,
    unless it contains some other $\Ic_{t'}$ or $\Ib$ in \ZeroCasesToDo.
    In either case, $\point$ will be listed for consideration as point of
    $\zeroes(\Ic_t)$ directly,
    or of $\zeroes(\Ic_{t'})\subset\zeroes(\Ic_t)$,
    or of $\zeroes(\Ib)\subset\zeroes(\Ic_{t})$,
    or of $\zeroes(\Ib)\subset\zeroes(\Ic_{t'})\subset\zeroes(\Ic_t)$
    (Steps~2.6.5 and 2.6.6).
    \\
    Notice that  $\point\in\zeroes(\Ia) \subset \zeroes(\ginKA)$, then
     $\point\in \zeroes(\Ic_t) = \zeroes(\Ic_t+\ginKA)$ 
\end{itemize}
This concludes the proof of the Claim.

Now, observe that the algorithm starts with $\ideal{0}$ in $K[A]$,
so,
%\cancella{for every $\point\in K^{N\!_A}$ there will be (at least) a corresponding GBasis in $\FinalCasesI$.}
%\anna{
  at the termination of the algorithm,
  each point $\point\in K^{N\!_A}$  is in  (at least)  a set
  $\pointsA$ in $\FinalCasesI$.
 % }
\end{proof}

  Although not optimal, we can prove that many redundant branches of
computations are avoided by checking ideal membership.
In fact, throughout the
computation, the set of ideals \ZeroCasesToDo is naturally inclusion-minimalized
and new additions to it are ``inclusion-increasing''. %{``inclusion-sorted''}.
More precisely,

    \begin{theorem}\label{thm:minimalized}
  Throughout the computation,
whenever an ideal $\Ib$ is added to \ZeroCasesToDo:

   \begin{enumerate}
  %% \item\label{PossiamoMinimalizzare\NewZeroCases}
  %%   If $\Ic_{T_1}\subset \Ic_{T_2}$ then $\Ic_{T_2}$ can be ignored
  %% \item\label{PossiamoButtareQuelliCheContengoUnoInZeroCasesToDo}
  %%   If  $\Ib\subset \Ic_{T}$ for some $\Ib\in$ZeroCasesToDo
  %%   then $\Ic_{T}$ can be ignored
    % \anna{
   \item\label{ZeroCasesToDo-interridotto}
     $\Ib\not\supseteq\Iz$  for any $\Iz$ in the current \ZeroCasesToDo
     \item\label{OneWay}
     $\Ib\not\subseteq\Iz$  for any $\Iz$ 
     ever listed in \ZeroCasesToDo
     \item\label{ZeroCasesToDo-interridotto}
       \ZeroCasesToDo is minimalized
  %   }
   \end{enumerate}
  %%  \cancella{
  %%    \begin{enumerate}
  %% %% \item\label{PossiamoMinimalizzare\NewZeroCases}
  %% %%   If $\Ic_{T_1}\subset \Ic_{T_2}$ then $\Ic_{T_2}$ can be ignored
  %% %% \item\label{PossiamoButtareQuelliCheContengoUnoInZeroCasesToDo}
  %% %%   If  $\Ib\subset \Ic_{T}$ for some $\Ib\in$ZeroCasesToDo
  %% %%   then $\Ic_{T}$ can be ignored
  %%   % \anna{
  %%  \item\label{ZeroCasesToDo-interridotto}
  %%    At any iteration, \ZeroCasesToDo is always minimalized wrt inclusion.
  %%    \item\label{OneWay}
  %%    No ideal added to \ZeroCasesToDo is contained in any ideal
  %%    previously listed in \ZeroCasesToDo.
  %% %   }
  %%    \end{enumerate}
  %%    }
\end{theorem}

\begin{proof}    %%   \ref{NonAbbiamoFattoCasiInutili}

    By induction.
    At the end of the first iteration \ZeroCasesToDo is either
  $\{\ideal{g}\}$ (Step 2.5.2) or   minimalized
  $\{ \Ic_t + \ginKA \mid  t {\in} \MBLPP  \text{ and }
  \Ic_t{\not\subseteq} \ginKA\}$
  (Step 2.6.5).
And all 3 statements are trivially satisfied.

Let's consider an iteration starting at Step 2.1, picking $\Ia$ from \ZeroCasesToDo;
\ZeroCasesToDo is minimalized by induction,
so $\Ia\not\subseteq \Iz$ for any $\Iz\in$ \ZeroCasesToDo.
We have $\Ia \subseteq \ginKA$ (Step 2.4).

  (1)
$\Ib\not\supseteq\Iz$ is explicitly checked in either cases:
$\Ib=\ginKA$ (Step~2.5.2)
or $\Ib\in$ \NewZeroCases (Step~2.6.6).

  (2)
If we add $\Ib=\ginKA$ in Step~2.5.2, then $\Ia \subsetneq \ginKA$.

If we add $\Ib$ from \NewZeroCases in Step~2.6.6, then
$\Ia\subseteq\ginKA\subsetneq\Ib$.

In either cases, $\Ib\not\subseteq\Iz$ for any ideal $\Iz$
     ever listed in \ZeroCasesToDo, because so was $\Ia$ (by
     inductive hypothesis), and because $\Ia\subsetneq\Ib$.

%\ref{ZeroCasesToDo-interridotto}.

     (3)
     In particular, by (1) and (2), it follows that 
     the current \ZeroCasesToDo is minimalized.

%%   \ref{PossiamoMinimalizzare\NewZeroCases}
%%   If $\Ic_{T_1}\subset \Ic_{T_2}$ then
%%   $\zeroes(\Ic_{T_2})\subseteq \zeroes(\Ic_{T_1})$
%% so   $\point\in\zeroes(\Ic_{T_2})$ is in  $\zeroes(\Ic_{T_1})$
%%   %$\Ic_{T_2}$ can be ignored
%%   \ref{PossiamoButtareQuelliCheContengoUnoInZeroCasesToDo}
\end{proof}

\begin{remark}\label{rem:iter-minimalization}
We just proved that many redundant computations are avoided thanks to the
  minimalizations (with respect to ideal inclusion) in Steps~2.6.5 and 2.6.6,
  and discarding ideals in Step~2.5.2.
  
Compared with the check $\zeroes(\Ia){\setminus} \zeroes(\Ib)\ne\emptyset$
described by Nabeshima in~\cite{Nabe2012} (within these
  algorithms equivalent to the condition~$\sqrt{\Ia} \ne\sqrt{\Ib}$),
our approach misses to detect
  some
redundant branches,  but it is computationally cheaper 
because it is based on \textit{ideal membership}, which can use many times
the same pre-computed \Groebner bases of the ideals, 
instead of computing the radicals, or the \textit{radical membership}
(which needs a specific \Groebner basis for each pair $\Ia,\Ib$ --
see Section~\ref{sec:Avoiding}).

Notice that we kept the %chose to do
check 
$\zeroes(\Ia){\setminus} \zeroes(\ginKA)\ne\emptyset$ in Step~2.5.
Choosing when and where to use this more effective+expensive
check, is indeed a delicate balance.
%We believe our approach has some advantage in the general case.
\end{remark}

\begin{remark}\label{rem:ginKA}
  In Step~2.5.2 we just append the ideal $\ginKA$ to \ZeroCasesToDo
    (if needed) and then pass to the next iteration,
  instead of %directly considering the recursive steps
  going on and consider the ideals given by
    the $\LC_X$ in $G$ (as in Step~2.6).
    By so doing, we ``waste'' the information included in the Gr\"obner
basis $G$ we just computed, but on the other hand, the fact that
$\ginKA$ is quite different form its ``ancestor'' $\Ia$
(it has a different vanishing set) make us believe that we'd better
reconsider $\ginKA$ in the ``whole picture'', as a new element in
the \ZeroCasesToDo list (if needed).
Moreover, before checking and appending it, we make its generators
square-free, so that, when it is processed in a following iteration
(Step~2.1), it is ``closer'' to its radical and may reduce the number
of the subsequent computations.
\end{remark}

%% EXAMPLE?
%% The iterative algorithm: demo

%% Consider
%% \begin{itemize}
%% \item circle with center $(0,0)$ and radius $1$
%% \item circle with center $(c_1,c_2)$ and radius $\sqrt{r}$
%% \end{itemize}
%% %  A questo spiegare che si puo' ottimizzare facendo i radicali
%% %  degli ideali ``=0'',
%% %
%% %  e semplicemente evitando di mettere doppioni.
%% %  
%% %  demo 2 
%%   %two circles (bis)  -- comparison of optimization
  
%% \begin{figure}[t]
%% \centering
%% \begin{tikzpicture}
%% \draw (0,0) node[v,label=above:{$ (  0,  0) $}](1){};
%% \draw (2,1) node[v,label=above:{$ (  c_1,  c_2) $}](2){};
%% \draw (-0.5,0.1) node[label=below:{$ 1 $}](3){};
%% \draw [densely dashed](-1,0) -- (0,0);
%% \draw (2.5,1.1) node[label=below:{$ \sqrt{r} $}](4){};
%% \draw [densely dashed](2,1) -- (3.5,1);
%% \draw (0,0) circle (1cm);
%% \draw (2,1) circle (1.5cm);
%% \end{tikzpicture}
%% \end{figure}

%%%%%%%%%%%%%%%%%%%%%%%%%%%%%%%%%%%%%%%%%%%%%%%%%%%%%%%%%%%%%%%%%%%%%%
\section{Timings}\label{sec:Timings}
\newcommand{\Z}{$\approx0$ s}

  In this table we compare our iterative algorithm, ``iter'',  with
  our best implementations of Nabeshima's algorithm, ``N'' and ``N$^*$'',
  in CoCoA-5.4.1v, running on a MacBookPro (i7, 2.3GHz).

  In ``N$^*$'' we disabled Nabeshima's ``RoughCheck'', because
  this operation is costly and may be intended as a post-processing applicable
  at the end of any algorithm for removing those segments whose support set is empty.
    Thus ``N$^*$'' takes less time and returns a higher number of segments.
  \medskip
  
  We used the examples listed in the papers \cite{Nabe2012}, \cite{KSW2010} and \cite{KSW2013},
  and we encountered some discrepancies, so we indicate both names.
  Moreover, for some examples, the time of the very first GB in CoCoA
is higher than the time obtained by the other authors for the whole
computation, thus we are not quite sure we are using the same
conditions/definitions. 
% the comparison with our implementation of Nabeshima's
% algorithm under the same conditions seems quite fair.
In particular, CoCoA could not compute the GBasis of  Nabeshima's
example S2 with $a>b>c>d$, so we changed the example into $b>c>d>a$.

\medskip

At the following link
  \begin{center}
     \small{
       \texttt{http://www.dima.unige.it/\~bigatti/data/ComprehensiveGroebnerSystems/}}
  \end{center}
the interested reader will find our
CoCoA-5 package, CGS.cpkg5 (which is part of the forthcoming
release CoCoA-5.4.2), together with the file containing
all examples (which we used to generate the table of timings).

\medskip
  To investigate the different routes followed by the two algorithms, we
  tracked the number of calls and the time spent for some crucial
  operations:

  \begin{description}
    \item[{\boldmath GB in $K[A,X]$}]  the Gr\"obner basis computations of
      $I+\Ia\in K[A,X]$ with an elimination ordering for $X$.
      This may be quite time consuming, especially for the very first
      computation: the Gr\"obner basis of the input ideal $I$.
      
    \item[{\boldmath GB in $K[A]$}]  the Gr\"obner basis computations of $\Ia\in
      K[A]$: this is information is extracted from its actual
      application within ``check $\mathfrak a\subseteq \mathfrak b$'' and
      ``check  $\Bbb V(\mathfrak a)\backslash \Bbb V(\mathfrak b){=}\emptyset$''.
    \item[{\boldmath check $\mathfrak a\subseteq \mathfrak b$}]
            only in our iterative algorithm:
      ideal membership $f\in \Ib$, using the Gr\"obner basis of $\Ib$
    \item[{\boldmath check  $\Bbb V(\mathfrak a)\backslash \Bbb
        V(\mathfrak b){=}\emptyset$}]
      (``Checking consistency'' in \cite{Nabe2012}
      radical membership $f\in \sqrt{\Ib}$,
      checking first if $f\in\Ib$ (using ``GB in $K[A]$''),
      and then using the Gr\"obner basis of $\Ib+\langle1-t{\cdot}f\rangle$.
    \item[{\boldmath product}]
only in Nabeshima's algorithm:
      the product of the ideals, $N$, passed recursively as second
      argument 
      %% (see Note at the end of the algorithm, page 254)
      %% (!!controllare rough check!!  -- ``In
      %% [9], Kapur-Sun-Wang gave great algorithms for checking “f is in
      %% the radical ideal'')
    \item[{\boldmath MB}]
      the computations of the minimal basis of the
      monomial ideal
      $\ideal{\LPP_X(f) \mid f\in GB} \; \subset K[X]$
      (Definition~\ref{def:MB})
\item[{\boldmath sqfr}]
  the computations for making 
  square-free the generators of the ideals
  in $K[A]$
  (Steps~2.5.2 and 2.6.5, and Remark~\ref{rem:ginKA})
  without computing the actual radical of the ideal.
\item[{last column}]
  final number of segments in
    the resulting CGS,   and the total time for its computation.
  \end{description}

In this table we observe that most of the time is spent in computing
the GB in $K[A,X]$ and the operations involved in removing redundant
computational branches (GB in $K[A]$, product of ideals, the checks
themselves).  The complete algorithm ``N'' also performs
``RoughCheck'' which  removes some useless segment from the final output.

\goodbreak

{\small
%20231214 h 17:08

\noindent 
\begin{tabular}{|l
    |@{\;}>{\scriptsize}r@{\;\,}>{\footnotesize}r@{\;}
    |@{\;}>{\scriptsize}r@{\;\,}>{\footnotesize}r@{\;}
    |@{\;}>{\scriptsize}r@{\;\,}>{\footnotesize}r@{\;}
    |@{\;}>{\scriptsize}r@{\;\,}>{\footnotesize}r@{\;}
    |@{\;}>{\scriptsize}r@{\;\,}>{\footnotesize}r@{\;}
    |@{\;}>{\scriptsize}r@{\;\,}>{\footnotesize}r@{\;}
    |@{\;}>{\scriptsize}r@{\;\,}>{\footnotesize}r@{\;}
   ||@{\;}>{\scriptsize}r@{\;\,}>{\footnotesize}r@{\;}
    |}
\hline
& \multicolumn{2}{@{\;}c}{GB in}
& \multicolumn{2}{@{\;}c}{GB in}
& \multicolumn{2}{@{\;}c}{check}
& \multicolumn{2}{@{\;}c}{check ${=}\emptyset$}
& \multicolumn{2}{@{\;}c}{} 
& \multicolumn{2}{@{\;}c}{} & \multicolumn{2}{c}{} &  &
\\
& \multicolumn{2}{@{\;}c}{$K[A,X]$}
& \multicolumn{2}{@{\;}c}{$K[A]$}
& \multicolumn{2}{@{\;}c}{$\mathfrak a\subseteq \mathfrak b$}
& \multicolumn{2}{@{\;}c}{$\Bbb V(\mathfrak a)\backslash \Bbb V(\mathfrak b)$}
& \multicolumn{2}{@{\;}c}{product} 
& \multicolumn{2}{@{\;}c}{MB}
& \multicolumn{2}{@{\;}c}{sqfr} & \# & time \\\hline
\multicolumn{17}{|l|}{\textbf{SuzukiSato 3}} \\\hline
iter & 30  &  0.1s & 175  &  0.0s & 419  &  0.0s &  30  &  0.0s & -   &   -  &  29  &  0.0s & 268  &  0.0s & 30 &  0.4s\\\hline
N*   & 27  &  0.1s & 161  &  0.0s &  -   &   -  & 187  &  0.0s & 159  &  0.0s &  26  &  0.0s & 242  &  0.0s & 28 &  0.2s\\\hline
N    & 27  &  0.1s & 172  &  0.0s &  -   &   -  & 299  &  0.1s & 271  &  0.0s &  26  &  0.0s & 242  &  0.0s & 19 &  0.3s\\\hline
\multicolumn{17}{|l|}{\textbf{SuzukiSato 4}} \\\hline
iter &  8  &  0.0s & 40  &  0.0s &  54  &  0.0s &   8  &  0.0s & -   &   -  &  8  &  0.0s & 83  &  0.0s & 8 &  0.1s\\\hline
N*   &  9  &  0.0s & 42  &  0.0s &  -   &   -  &  50  &  0.0s & 40  &  0.0s &  9  &  0.0s &  91  &  0.0s & 9 &  0.1s\\\hline
N    &  9  &  0.0s & 51  &  0.0s &  -   &   -  &  80  &  0.0s & 70  &  0.0s &  9  &  0.0s &  91  &  0.0s & 8 &  0.1s\\\hline
\multicolumn{17}{|l|}{\textbf{Nabeshima F8}} \\\hline
iter & 24  &  0.1s & 109  &  0.0s & 306  &  0.0s &  24  &  0.0s & -   &   -  &  24  &  0.0s & 215  &  0.0s & 24 &  0.3s\\\hline
N*   & 23  &  0.1s & 107  &  0.0s &  -   &   -  & 129  &  0.1s & 105  &  0.0s &  23  &  0.0s & 207  &  0.0s & 23 &  0.3s\\\hline
N    & 23  &  0.1s & 115  &  0.0s &  -   &   -  & 195  &  0.1s & 171  &  0.0s &  23  &  0.0s & 207  &  0.0s & 18 &  0.4s\\\hline
\multicolumn{17}{|l|}{\textbf{Nabeshima E5 }} \\\hline
iter &  5  &  0.0s & 10  &  0.0s &   8  &  0.0s &   5  &  0.0s & -   &   -  &   4  &  0.0s &  48  &  0.0s & 5 &  0.1s\\\hline
N*   &  4  &  0.0s & 10  &  0.0s &  -   &   -  &  13  &  0.1s &  8  &  0.0s &   4  &  0.0s &  42  &  0.0s & 8 &  0.2s\\\hline
N    &  4  &  0.0s & 12  &  0.0s &  -   &   -  &  17  &  0.1s & 12  &  0.0s &   4  &  0.0s &  42  &  0.0s & 8 &  0.2s\\\hline
\multicolumn{17}{|l|}{\textbf{Nabeshima S2 -- modified --> b,c,d,a }} \\\hline
iter &  7  &  7.2s & 20  &  0.0s &  22  &  0.0s &   7  &  0.0s & -   &   -  &  7  &  0.0s & 129  &  0.0s & 7 &  7.4s\\\hline
N*   &  8  &  6.9s & 24  &  0.0s &  -   &   -  &  31  &  0.2s & 22  &  0.0s &  8  &  0.0s & 145  &  0.0s & 8 &  7.2s\\\hline
N    &  8  &  6.8s & 26  &  0.0s &  -   &   -  &  45  &  0.5s & 36  &  0.0s &  8  &  0.0s & 145  &  0.0s & 8 &  7.4s\\\hline
\multicolumn{17}{|l|}{\textbf{Nabeshima S3 = S1 KapurSunWang }} \\\hline
iter & 11  &  0.5s & 34  &  0.1s &  64  &  0.0s &  11  &  0.0s & -   &   -  &  11  &  0.0s & 186  &  0.1s & 11 &  0.8s\\\hline
N*   & 14  &  0.4s & 45  &  0.1s &  -   &   -  &  58  &  1.7s & 43  &  1.0s &  14  &  0.0s & 177  &  0.1s & 14 &  3.5s\\\hline
N    & 14  &  0.4s & 48  &  0.1s &  -   &   -  &  84  &  2.5s & 69  &  1.6s &  14  &  0.0s & 177  &  0.1s & 10 &  5.0s\\\hline
\multicolumn{17}{|l|}{\textbf{(similar to Nabeshima S4) S2 KapurSunWang }} \\\hline
iter &  4  &  0.0s &  7  &  0.0s &   2  &  0.0s &   4  &  0.1s & -   &   -  &   3  &  0.0s &  58  &  0.0s & 4 &  0.2s\\\hline
N*   &  3  &  0.0s &  7  &  0.0s &  -   &   -  &   9  &  0.2s &  5  &  0.0s &   3  &  0.0s &  57  &  0.0s & 4 &  0.3s\\\hline
N    &  3  &  0.0s &  8  &  0.0s &  -   &   -  &  11  &  0.4s &  7  &  0.0s &   3  &  0.0s &  57  &  0.0s & 4 &  0.5s\\\hline
\multicolumn{17}{|l|}{\textbf{Nabeshima S4}} \\\hline
iter & 18  &  8.6s & 51  &  0.1s &  83  &  0.0s &  18  &  0.0s & -   &   -  &  18  &  0.0s & 236  &  0.1s & 18 &  8.9s\\\hline
N*   & 21  &  9.0s & 62  &  0.1s &  -   &   -  &  82  &  7.3s & 60  & 10.6s &  21  &  0.0s & 323  &  0.1s & 21 & 28.6s\\\hline
N    & 21  &  9.0s & 65  &  0.1s &  -   &   -  & 119  &  9.5s & 97  & 11.9s &  21  &  0.0s & 323  &  0.1s & 15 & 32.3s\\\hline
\multicolumn{17}{|l|}{\textbf{(similar to Nabeshima S5) S3 KapurSunWang }} \\\hline
iter & 63  &  0.8s & 277  &  0.2s & 831  &  0.0s &  63  &  0.1s & -   &   -  & 60  &  0.1s & 975  &  0.2s & 63 &  2.0s\\\hline
N*   & 31  &  0.4s & 133  &  0.1s &  -   &   -  & 163  &  7.8s & 131  &  9.0s &  28  &  0.0s & 425  &  0.1s & 31 & 19.2s\\\hline
N    & 31  &  0.4s & 138  &  0.1s &  -   &   -  & 226  & 19.5s & 194  &  9.2s &  28  &  0.0s & 425  &  0.1s & 18 & 31.1s\\\hline
\multicolumn{17}{|l|}{\textbf{Nabeshima S5}} \\\hline
iter & 79  &  0.8s & 416  &  0.2s & 1315  &  0.0s &  79  &  0.1s & -   &   -  & 79  &  0.1s & 1202  &  0.2s & 79 &  2.5s\\\hline
N*   & 34  &  0.4s & 174  &  0.1s &  -   &   -  & 207  &  7.6s & 172  &  8.9s & 34  &  0.0s & 485  &  0.1s & 34 & 18.8s\\\hline
N    & 34  &  0.4s & 179  &  0.1s &  -   &   -  & 299  & 19.4s & 264  &  9.4s & 34  &  0.0s & 485  &  0.1s & 20 & 31.2s\\\hline
\multicolumn{17}{|l|}{\textbf{Nabeshima P3P = P3P KapurSunWang}} \\\hline
iter & 22  &  1.1s & 60  &  0.2s &  94  &  0.0s &  22  &  4.0s & -   &   -  &  21  &  0.0s & 510  &  0.3s & 22 &  5.9s\\\hline
N*   & 18  &  1.0s & 49  &  0.2s &  -   &   -  &  66  &  4.7s & 47  &6.7s &  16  &  0.0s & 329  &  0.2s & 20 & 13.9s\\\hline
N    & 18  &  1.1s & 52  &  0.2s &  -   &   -  &  95  &  7.4s & 76  &  7.2s &  16  &  0.0s & 329  &  0.2s & 19 & 17.3s\\\hline
\end{tabular}

} %small

%%%%%%%%%%%%%%%%%%%%%%%%%%%%%%%%%%%%%%%%%%%%%%%%%%%%%%%%%%%%%%%%%%%%%%
\section{Conclusion and future work}\label{sec:Future}

We believe that tackling the problem with an iterative
approach allows to have a better 
perspective on the computation: having the
  cases yet to be considered altogether in a single list makes it
  possible to compare them with the looser but faster ideal
  membership, instead of keeping
  track of them through the recursion with product of ideals and the
  costlier radical membership.
Moreover, we expect that 
  further new strategies may be developed for optimizing the computations.

  %Some of the ideas in this article may benefit from some extra work.
  In particular, we think that there are two main areas that can be studied further:

\begin{itemize}
\item Strategies for picking the next ideal $\Ia\subseteq K[A]$ to be considered
  (Step~2.1, function  \texttt{ChooseVanishing}).

  Currently we randomly pick one out of those with maximum dimension.

\item Study possible criteria to minimize {\ZeroCasesToDo}, in order
  to reduce even further the cases that we analyze.
  
  Currently we keep in the list the minimal ones (Steps~2.5.2
  and~2.6.6), so that we have ``more zeroes'', and we make their
  generators square-free.
  But ideally with should keep the radicals of these ideals, or just keep
  the larger when two ideals sharing the same radical. Thus, finding a
  cheap check to detect ``more zeroes or closer to the radical'' should
  be a further improvement.
\end{itemize}

%% Our questions:
%% \begin{enumerate}
%% \item[$\mathcal{Q}2$] Please help on\\
%%  ``cost to make the CGS reduced is very small'' \\
%%   ``the cost to check if a segment is redundant is very high''.
%% \end{enumerate}

%%------------------------------------------------------------------------------
\backmatter
%%------------------------------------------------------------------------------

\bmhead{Acknowledgements}
%\begin{acknowledgements}
The authors would like to thank Katsusuke Nabeshima, Yosuke Sato and Miwa Taniwaki.
%\end{acknowledgements}

This research was partly supported by the ``National Group for
Algebraic and Geometric Structures, and their Applications'' (GNSAGA-INdAM).




\end{document}